\documentclass[reqno,10pt]{amsart} 
\usepackage[colorlinks = true, linkcolor=blue,
            urlcolor=red,
            citecolor=olive]{hyperref}
\usepackage[utf8]{inputenc}
\usepackage[T1]{fontenc}
\usepackage{amsmath}

\usepackage{mathtools}      % improves amsmath
\usepackage{mathabx}        % makes many symbols (as $\leq$) more beautiful; must be loaded after mathtools 

\usepackage{tikz}
\usepackage{tikz-3dplot}

\usepackage{amssymb}
\usepackage{amsfonts}
\usepackage{graphicx}

\usepackage{multicol}

\usepackage[colorinlistoftodos]{todonotes}

\tikzset{surface/.style={draw=blue!70!black, fill=blue!40!white, fill opacity=.6}}
\newcommand{\coneback}[4][]{
  \draw[canvas is xy plane at z=#2, #1] (45-#4:#3) arc (45-#4:225+#4:#3) -- (O) --cycle;
  }
\newcommand{\conefront}[4][]{
  \draw[canvas is xy plane at z=#2, #1] (45-#4:#3) arc (45-#4:-135+#4:#3) -- (O) --cycle;
  }
\tikzset{circ/.style={shape=circle, inner sep=1.5pt, draw, node contents=}}

\hypersetup{bookmarksdepth  =  3} % Depth of sections/subs... to have bookmark links in the pdf;
\newtheorem{thm}{Theorem}[section]
\newtheorem{lem}[thm]{Lemma}
\newtheorem{prop}[thm]{Proposition}
\newtheorem{cor}[thm]{Corollary}

\newtheorem{defn}[thm]{Definition}

\newtheorem{rem}[thm]{Remark}

\newtheorem{nota}[thm]{Notation}

\newcommand{\R}{\mathbb{R}}

\DeclareMathOperator{\im}{Im}

\DeclareMathOperator{\Ker}{Ker}
\DeclareMathOperator{\PSL}{PSL}
\DeclareMathOperator{\SL}{SL}
\DeclareMathOperator{\GL}{GL}
\DeclareMathOperator{\Hom}{Hom}

\DeclareMathOperator{\Grass}{Grass}
\DeclareMathOperator{\Sym}{Sym}

\DeclareMathOperator{\Sp}{Sp}

\newcommand{\mathvec}[1]{\mathrm{#1}}

\makeatletter
\DeclareRobustCommand*{\bfseries}{%
  \not@math@alphabet\bfseries\mathbf
  \fontseries\bfdefault\selectfont
  \boldmath
}
\makeatother

\newcommand{\ie}{i.e. }

\title{Maximal and Borel Anosov representations into $\Sp(4,\R)$}

\copyrightinfo{2021}{Colin Davalo}

\author{Colin Davalo}
\address{Mathematisches Institut, Ruprecht-Karls Universität Heidelberg, Im Neuenheimer Feld 205, 69120 Heidelberg, Germany}
\email{cdavalo@mathi.uni-heidelberg.de}

\date{\today}
\begin{document}

\begin{abstract}

We prove that any Borel Anosov representation of a surface group into $\Sp(4,\R)$ that has maximal Toledo invariant must be Hitchin.
 We 	also prove that a representation of a surface group into $\Sp(2n,\R)$ that is $\{n-1,n\}$-Anosov is maximal if and only if it satisfies the hyperconvexity property $H_n$. 
 
\end{abstract}
\maketitle

\tableofcontents

\section{Introduction.}

In recent years, the study of discrete hyperbolic subgroups of semisimple Lie groups drew a lot of attention with the introduction of the concept of Anosov representations. Labourie introduced the notion of Anosov representations of the fundamental group $\Gamma_g$ of a closed surface of genus $g\geq 2$ into semi-simple Lie groups \cite{LabourieHyperconvexity}, generalizing the notion of convex-cocompact representations of hyperbolic groups into Lie groups of rank $1$ \cite{Guichard_2012}.
Representations satisfying this dynamical property have discrete image, are faithful and they form an open subset of the space of representations.
%The definition from Labourie was later extended to general Gromov hyperbolic groups instead of surface groups \cite{GuichardWienhardDiscontinuity}, and 
The notion of Anosov representations plays an important role in the study of higher Teichmüller spaces \cite{wienhard2018invitation}, and in particular \emph{Hitchin} representations and \emph{maximal} representations are Anosov \cite{LabourieHyperconvexity}, \cite{Burger2005MaximalRO}.

\medskip

In a semi-simple Lie group of rank at least $2$, several distinct notions of Anosov representations can be defined depending on the choice  of a conjugacy class of parabolic subgroup $P<G$. 
When $G=\PSL(N,\R)$ the minimal parabolic subgroup, \ie the parabolic subgroup that admits no proper parabolic subgroups, is called the \emph{Borel subgroup}.
A representation which is Anosov with respect to the Borel subgroup is also Anosov with respect to any other parabolic subgroup.
Such a representation is called \emph{Borel Anosov}. 

\medskip

%\begin{defn}
%
%Let $N\geq 2$ be an integer. Let $\eta :\SL(2,\mathbb{R})\to \SL(N,\mathbb{R})$ be the irreducible representation, which is unique up to conjugation by an element of $\GL(N,\R)$.
%
%
%A representation $\rho :\Gamma_g \to \SL(2,\mathbb{R})$ is \emph{Fuchsian} if is discrete and faithful.
%
%
%A representation $\rho :\Gamma_g \to \SL(N,\mathbb{R})$ is $N$-\emph{Fuchsian} if it can be written $\rho=\eta\circ \rho_0$ with $\rho_0:\Gamma_g\to \SL(2,\mathbb{R})$ a Fuchsian representation.
%
%
%A representation $\rho :\Gamma_g \to \SL(N,\mathbb{R})$ is \emph{Hitchin} if it is in the connected component in $\Hom(\Gamma_g,\SL(N,\mathbb{R}))$ of a representation that is conjugate in $\GL(N,\R)$ to a $N$-Fuchsian representation.
%\end{defn}

Hitchin representations are representations $\rho:\Gamma\to \SL(N,\R)$ in the same connected component as the composition $\eta\circ \rho_0$ of a Fuchisan representation $\rho_0:\Gamma_g\to \SL(2,\R)$ and the irreducible representation $\eta:\SL(2,\R)\to \SL(N,\R)$.
Hitchin showed that the space of Hitchin representations up to conjugation by elements in $\GL(N,\R)$ is a topological ball of dimension $(N^2-1)(2g-2)$, which generalizes a property of Teichmüller space \cite{HITCHIN1992449}.
Labourie showed that these representations are Borel Anosov, and in particular are discrete and faithful. The space of Hitchin representations thus defines a \emph{higher Teichmüller space}, as defined in \cite{wienhard2018invitation}. There is a notion of Hitchin representations $\rho:\Gamma_g\to G$ for every simple \emph{split} Lie group
$G$, and for $G=\Sp(2n,\R)$ a representation is Hitchin if and only if its inclusion in $\SL(2n,\R)$ is Hitchin.

\medskip

If $G$ is a semi-simple \emph{Hermitian} Lie group of \emph{tube type}, for instance $\Sp(2n,\R)$, one can define another generalization of the Teichmüller space using the \emph{Toledo invariant}. The Toledo invariant associates to every representation an integer whose sign depends on the orientation of the Gromov boundary $\partial \Gamma_g$ of $\Gamma$ and whose absolute value is bounded by a generalized Schwarz-Milnor inequality. A representation $\rho:\Gamma_g\to G$ is \emph{maximal} if its Toledo invariant is maximal. The space of maximal representations is a union of connected components of discrete and faithful representations of surface groups \cite{BURGER2003387}. It is another kind of higher Teichmüller space. Maximal representations are Anosov with respect to a parabolic subgroup $P<G$ that is maximal \cite{Burger2005MaximalRO}, \ie that is not properly contained in any other parabolic subgroup. 

\medskip

The group $\Sp(4,\R)$ is the smallest group apart from $\SL(2,\R)$ that is both split and Hermitian of tube type. Maximal representations $\rho:\Gamma_g\to \Sp(4,\R)$ are $\lbrace 2\rbrace$-Anosov, \ie are Anosov with respect to the stabilizer of a Lagrangian.
Every Hitchin representation whose image lie in $\Sp(4,\R)\subset \SL(4,\R)$ is also maximal for one of the orientations of $\partial \Gamma_g$ \cite{Burger2005MaximalRO}, and is $\lbrace 1,2\rbrace$-Anosov, \ie Borel Anosov. This is stronger than just being $\lbrace 2\rbrace$-Anosov. A natural question, see for instance Canary \cite[{Question 50.7}]{Canary}, is the following: are Hitchin representations the only maximal and Borel Anosov representations in $\Sp(4,\R)$ ? In this paper we answer in the affirmative.

\begin{thm}
\label{thm:mainSp4}
Every representation $\rho:\Gamma_g\to \Sp(4,\R)$ of a surface group that is maximal and Borel Anosov is Hitchin.

\end{thm}

Borel Anosov representations come with	a natural \emph{boundary map} from the Gromov boundary $\partial \Gamma_g$ of $\Gamma_g$ into the space of full flags. The set of Hitchin representations was characterized by Labourie \cite{LabourieHyperconvexity} and Guichard \cite{GuichardHyperconvexity} as the set of Borel Anosov representations whose associated boundary map is \emph{hyperconvex}. In order to prove Theorem \ref{thm:mainSp4}, we prove that maximal and Borel Anosov representations of surface groups in $\Sp(4,\R)$ are hyperconvex. 
 
 More generaly we consider the hyperconvexity conditions $H_k$ in any dimension, which are conditions satisfied by an open set of $\lbrace k-1,k,k+1\rbrace$-Anosov representations, see Definition \ref{defn:Hn}. Pozzetti-Sambarino-Wienhard \cite{pozzetti2020conformality} showed that if a representation satisfies property $H_k$, then the boundary map associated with $\rho$ in the Grassmanian of $k$-planes has $\mathcal{C}^1$ image and has a derivative prescribed by the boundary maps in the Grassmanians of $(k-1)$-planes and $(k+1)$-planes, see Theorem \ref{thm:SmoothBoundary}. We recall this result and some of its consequences in Section \ref{sec:Smooth}.
 
\medskip 
 
 In Section \ref{sec:MaximalityHn} we use this result to prove the following characterization of maximal representation that are $\{n-1,n\}$-Anosov in $\Sp(2n,\R)$.

\begin{thm}
\label{thm:mainSp2n}
Let $\rho:\Gamma_g\to \Sp(2n,\R)$ be an $\{n-1,n\}$-Anosov representation of a surface group $\Gamma_g$. The representation $\rho$ satisfies property $H_{n}$ if and only if it is maximal for some orientation of $\partial\Gamma_g$.

\end{thm}

Maximal representation can be characterized as Anosov representations that admit an equivariant positive boundary map in the space of Lagrangians, see Definition \ref{defn:MaximalTriplet}. Our result is a new link between positivity for maximal representations and hyperconvexity of boundary maps, for $\{n-1,n\}$-Anosov representations. On the one hand we see that maximality forces property $H_n$. On the other hand we show that property $H_n$ implies positivity of the $n$-th boundary map combining the characterization of the tangents to boundary maps \cite{pozzetti2020conformality} together with the observation that a $\mathcal{C}^1$ curve whose derivative stays in a cone also stays in a the cone.

\begin{figure}
\begin{center}
\tdplotsetmaincoords{80}{45}
\tdplotsetrotatedcoords{-90}{180}{-90}
\usetikzlibrary{patterns}

\begin{tikzpicture}[tdplot_main_coords, scale=.6]

  \coordinate (O) at (0,0,0);
 
  \coneback[surface]{3}{3}{10}
  %\conefront[surface]{-3}{-3}{10}
\draw plot[variable=\t,domain=-1.5:1.5,samples=73,smooth] 
(1.73205080757*\t^2,\t-\t^3,\t+\t^3);
 \fill (1.73205080757,0,2) circle (2pt);
 \conefront[surface]{3}{3}{10};

  \node[left] at (1.73205080757,0,2){$y^2_{\rho}$};
  \fill (0,0,0) circle (2pt);
d  \node[left] at (-0.2,0,0){$x^2_{\rho}$};
  %\coneback[surface]{-3}{-3}{10};
\end{tikzpicture}
\begin{tikzpicture}[scale=.7]
\fill[cyan] (89:2) arc (89:226:2) -- cycle;
\draw (0,0) circle (2);
\node[below right] at (0,0) {$[y^2_\rho]$};
\fill (0,0) circle (2pt);
\fill[red] (90:2) circle (2pt);
\node[above] at (90:2) {$[x_\rho^1]$};
\draw[blue,fill=white] (270:2) circle (2pt);
\node[below] at (270:2) {$[z_\rho^3]$};
\fill[red] (225:2) circle (2pt);
\node[above right] at (45:2) {$[y_\rho^3]$};
\node[below left] at (225:2) {$[y_\rho^1]$};

\draw[dotted] (225:2) -- (45:2);
\draw[->] (0,0) -- (225:0.5);
\draw[blue,fill=white] (45:2) circle (2pt);
\end{tikzpicture}
\end{center}

\caption{The second boundary map for a Hitchin representation, and the main argument of the proof of Theorem \ref{thm:mainSp4}.}
\label{fig:introVeronese}

\end{figure}

%\begin{figure}[!ht]
%\usetikzlibrary{patterns}
%\begin{center}
%\begin{tikzpicture}[scale=.7]
%\fill[cyan] (89:2) arc (89:226:2) -- cycle;
%\draw (0,0) circle (2);
%
%\node[below right] at (0,0) {$[y^2_\rho]$};
%\fill (0,0) circle (2pt);
%\fill[red] (90:2) circle (2pt);
%\node[above] at (90:2) {$[x_\rho^1]$};
%\draw[blue,fill=white] (270:2) circle (2pt);
%\node[below] at (270:2) {$[z_\rho^3]$};
%\fill[red] (225:2) circle (2pt);
%\node[above right] at (45:2) {$[y_\rho^3]$};
%\fill[red] (170:2) circle (2pt);
%\node[left] at (170:2){$[w_\rho^1]$};
%\node[below left] at (225:2) {$[z_\rho^1]$};
%
%\draw[dotted] (225:2) -- (45:2);
%\draw[->] (0,0) -- (225:0.5);
%\draw[blue,fill=white] (45:2) circle (2pt);
%\end{tikzpicture}
%\end{center}
%\caption{Main argument of the proof of Theorem \ref{thm:mainSp4}.}
%\label{fig:introIllustrationConvex}
%\end{figure}

\medskip

In order to prove Theorem \ref{thm:mainSp4} we use Theorem \ref{thm:mainSp2n} and prove in Section \ref{sec:H2toH1} that Borel Anosov representations $\rho:\Gamma_g\to \Sp(4,\R)$ which satisfy property $H_2$ also satisfy property $H_1$. 

For that we project the boundary map onto the parallel tube in the symmetric space between two Lagrangians in the boundary curve. Concretely, given $3$ points $(x,y,z)$ in the Gromov boundary of $\Gamma_g$ we consider their  full flags associated via the boundary map $(x^1_\rho,x^2_\rho,x^3_\rho)$, $(y^1_\rho,y^2_\rho,y^3_\rho)$, $(z^1_\rho,z^2_\rho,z^3_\rho)$  and construct $4$ points in the circle $\mathbb{P}(x^2_\rho)$ by projecting the lines $x^1_\rho$, $y^1_\rho$ and intersecting the hyperplanes $z_\rho^3$ and $y_\rho^3$, yielding 4 points on the boundary of a copy of the hyperbolic plane. The Lagrangian $y^2_\rho$ defines a point in the interior of this hyperbolic plane, see Figure \ref{fig:introVeronese}. 

We distinguish two possible configurations of these projections, one of which implies property $H_1$. To rule out the other configuration, we use again that the second boundary map has $\mathcal{C}^1$ image to show that the projection of the second boundary map must stay in a smaller convex cone, colored in the picture. 

This leads to a contradiction as the point corresponding to $y^2_\rho$ must lie in the geodesic joining the ideal points corresponding to $y_\rho^1$ and $y_\rho^3$, since $y^1_\rho\subset y^2_\rho\subset y^3_\rho$. This geodesic is disjoint from the convex if the four points are ordered as in the picture.

In Section \ref{sec:Hyperconvex} we recall results from Labourie and Guichard, implying that a Borel Anosov representation in $\Sp(4,\R)$ that satisfies property $H_1$ and $H_2$ is Hitchin.

\medskip

We hope that such geometric argument will be useful to rule out the existence of other kinds of Anosov representations.

%\begin{figure}[!ht]
%\usetikzlibrary{patterns}
%\begin{center}
%\begin{tikzpicture}
%\fill[cyan] (89:2) arc (89:226:2) -- cycle;
%\draw (0,0) circle (2);
%
%
%\fill (0,0) circle (2pt);
%\node[below right] at (0,0) {$\left[q_{x,z}^y\right]$};
%\fill[red] (90:2) circle (2pt);
%\node[above] at (90:2) {$\iota(\psi(x))$};
%\draw[blue,fill=white] (270:2) circle (2pt);
%\node[below] at (270:2) {$\iota(z^3_\rho\cap x^2_\rho)$};
%\fill[red] (225:2) circle (2pt);
%\node[above right] at (45:2) {$\iota(y^3_\rho\cap x^2_\rho)$};
%\fill[red] (170:2) circle (2pt);
%\node[left] at (170:2){$\iota(\psi(w))$};
%\node[below left] at (225:2) {$\iota(\psi(y))$};
%
%\draw[dotted] (225:2) -- (45:2);
%\draw[->] (0,0) -- (225:0.5);
%\draw[blue,fill=white] (45:2) circle (2pt);
%\end{tikzpicture}
%\end{center}
%\caption{$\mathbb{P}(\mathcal{Q}(x^2_\rho))$ and the convex $\mathbb{P}(C)$ from the proof of Lemma \ref{lem:Order4Points}.}
%\label{fig:IllustrationConvex}
%\end{figure}

\subsection*{Aknowledgments.}

I would like to thank Beatrice Pozzetti for exposing this problem to me, for discussing it and for her precious reviews of previous versions of the paper. The author was funded through the DFG Emmy Noether project 427903332 of B. Pozzetti. 

\section{Anosov representations.}
\label{sec:Anosov}

Let $\Gamma_g$ denote the fundamental group of a closed orientable surface of genus $g\geq 2$. This is an hyperbolic group in the sense of Gromov, and we will denote by $\partial \Gamma_g$ its Gromov boundary, which is a topological circle.

\medskip

Let $N\geq 2$ be an integer. Let us fix some Euclidean structure on $\mathbb{R}^{N}$, and for every element $M\in \SL(N,\mathbb{R})$ denote by $\sigma_1(M)\geq \sigma_2(M)\geq \cdots \geq \sigma_N(M)$ the singular values of $M$ in non-decreasing order. Given $\gamma\in \Gamma_g$ we will denote by $|\gamma|_w$ the word length of $\gamma$ with respect to some fixed finite generating set of $\Gamma_g$. 

The following definition is not the original one, but a characterization due to Kapovich-Leeb-Porti \cite{kapovichanosov} and Bochi-Potrie-Sambarino \cite{Bochi_2019}.

\begin{defn}[{\cite[Section 4]{Bochi_2019}}]
A representation $\rho$ of $\Gamma_g$ into $\SL(N,\mathbb{R})$ is $\Theta$-Anosov with $\Theta\subset \{1,\cdots,N\}$ if there exists some constants $C,\alpha>0$ such that for all $\gamma\in \Gamma_g$ and $k\in \Theta$ :
$$\frac{\sigma_{k+1}\left(\rho(\gamma)\right)}{\sigma_{k}\left(\rho(\gamma)\right)}\leq Ce^{-\alpha |\gamma|_w}.$$

If a representation is Anosov with respect to $\Delta=\{1,\cdots,N\}$, then it is called \emph{Borel Anosov}.
\end{defn}

\begin{rem}
If a representation is $\Theta$-Anosov then it is automatically $\Theta'$-Anosov for all $\Theta'\subset \Theta$.
\end{rem}

For a general semi-simple Lie group $G$, the Anosov property depends on a subset of the set of simple roots, or equivalently of a conjugacy class of parabolic subgroups. Here we identified the set $\Delta$ of simple roots of the simple Lie group $\SL(N,\R)$ with the set $\{1,\cdots, N-1\}$.

\medskip

Boundary maps are important objects naturally associated to an Anosov representation.

\begin{thm}[{\cite{Guichard_2012},\cite{Bochi_2019}}]
\label{thm:CaracterizationBoundaryMaps}
Let $\rho:\Gamma_g \to \SL(N,\mathbb{R})$ be a $\{k\}$-Anosov representation. Let $\Grass(k,N)$ be the Grassmannian of $k$-dimensional subspaces in $\R^N$. There exists a unique continuous $\rho$-equivariant map $\xi_\rho^k:\partial \Gamma_g\to \Grass(k,N)$ that is \emph{dynamic preserving}, i.e for all element $\gamma\in \Gamma_g$ if $\gamma^+$ is the unique attracting fixed point of $\gamma$ in $\partial \Gamma_g$ then $\xi_\rho^k(\gamma^+)$ is the unique attracting fixed point of $\rho(\gamma)$ in $\Grass(k,N)$.

\end{thm}

The property of being dynamic preserving determines $\xi_\rho^k$, since the set of attracting fixed point of elements of $\Gamma_g$ is dense in $\partial \Gamma_g$.

\begin{nota}

\label{nota:NotationBoundarymaps}
For any $\{k\}$-Anosov representation and any $x\in\partial \Gamma_g$ we will write $x_\rho^k:= \xi_\rho^k(x)$ as in \cite{pozzetti2020conformality} to make expressions involving boundary maps lighter. We will still keep the notation $\xi^k_\rho$ to denote the boundary map itself. As a convention $x^0_\rho=\{0\}$ and $x^N_\rho=\R^N$ for any $x\in \partial \Gamma_g$.
\end{nota}

Boundary maps satisfy additional properties: they are \emph{transverse} and \emph{compatible}.

\begin{prop}[{\cite{Guichard_2012},\cite{Bochi_2019}}]
\label{prop:TransversalityBoundaryMaps}
Let $\rho:\Gamma_g \to \SL(N,\mathbb{R})$  be a $\{k\}$-Anosov representation. The representation $\rho$ is also $\{N-k\}$-Anosov, and for every pair $x,y\in \partial \Gamma_g$ of distinct points, $x_\rho^k$ and $y_\rho^{N-k}$ are transverse \emph{(transversality)}.
If $\rho$ is $\{k,\ell\}$-Anosov with $k\leq \ell$ then $x_\rho^k\subset x_\rho^\ell$ for all $x\in \partial \Gamma_g$ \emph{(compatibility)}.
\end{prop}

As a consequence the image of boundary maps at two different point are in  general position.

\begin{cor}
Let $k,\ell\geq 1$. Let $x,y\in \partial \Gamma_g$ be distinct points :
$$\dim\left(x^k_\rho\cap y_\rho^\ell\right)=\max(k+\ell-N,0).$$
\end{cor}

Let us assume now that $N=2n$ is even and let us fix a symplectic form $\omega$ on $\R^{2n}$. Consider the subgroup $\Sp(2n,\mathbb{R})\subset \SL(2n,\mathbb{R})$ consisting of elements which preserve $\omega$: representations into $\Sp(2n,\R)$ can be seen as particular examples of representations into $\SL(2n,\R)$. The boundary maps of Anosov representations whose images lie in $\Sp(2n,\mathbb{R})$ have some additional properties.

\begin{lem}
\label{lem:ParticularOrthogonality}
Let $\rho:\Gamma_g\to \Sp(2n,\R)$ be a $\{k\}$-Anosov representation. For any $x\in \partial\Gamma_g$, $\left(x_\rho^k\right)^\perp=x_\rho^{2n-k}$, where $\left(x_\rho^k\right)^\perp$ is the orthogonal of $x_\rho^k$ with respect to $\omega$. In particular if $k\leq n$ the space $x_\rho^k$ is isotropic.

\end{lem}

\begin{proof}

The orthogonality condition holds for a closed $\Gamma_g$-equivariant 
subset of $\partial \Gamma_g$. Since the action of $\Gamma_g$ is minimal on $\partial \Gamma_g$ (\cite{HypRef} Proposition 4.2), it is sufficient to check it for a single point. Let $x$ be the attracting fixed point of an element $\gamma\in \Gamma_g$, so $x_\rho^k$ is the unique attracting fixed $k$-dimensional subspace of $\rho(\gamma)$ and $x_\rho^{n-k}$ the unique attracting fixed $(2n-k)$-dimensional subspace of $\rho(\gamma)$. 

Since $\gamma\in \Sp(2n,\mathbb{R})$, it maps any subspace $V^\perp$ for $V\subset \R^{2n}$ to $(\gamma\cdot V)^\perp$. Hence $\left(x_\rho^k\right)^\perp$ is an attracting fixed point for the action of $\gamma$ on the space of $(2n-k)$-dimensional subspaces.
 Therefore $\left(x_\rho^k\right)^\perp=x_\rho^{2n-k}$. If $k\leq n$, then $x_\rho^{2n-k}\subset\left(x_\rho^{k}\right)^\perp $ and hence $x_\rho^k$ is isotropic.

\end{proof}

\section{Charts of the space of Lagrangians and maximality.}

\label{sec:LinAlg}

Recall that we fixed a symplectic structure $\omega$ on $\R^{2n}$. Let $\mathcal{L}_n$ be the space of \emph{Lagrangians} in $\R^{2n}$, \ie the space of $n$-dimensional subspaces of $\R^{2n}$ on which $\omega$ vanishes. Let $P,Q\in \mathcal{L}_n$ be two transverse Lagrangians, \ie with trivial intersection.

\begin{defn}
A linear map $u$ between $P$ and $Q$ is \emph{symmetric} (with respect to $\omega)$ if for all $\mathvec{v},\mathvec{w}\in P$:
$$\omega\left(\mathvec{v},u(\mathvec{w})\right)= \omega\left(\mathvec{w},u(\mathvec{v})\right).$$

The space of symmetric linear maps $u$ from $P$ to $Q$ will be denoted by $\Sym_{P,Q}$.  
\end{defn}

For $Q\in \mathcal{L}_n$ let $U_Q$ be the set of Lagrangians transverse to $Q$. The open sets $(U_Q)_{Q\in \mathcal{L}_n}$ form an open covering of $\mathcal{L}_n$. Given a Lagrangian $P$ transverse to the Lagrangian $Q$, we get an identification of $U_Q$ with the vector space $\Sym_{P,Q}$. This provides a family of linear charts of $\mathcal{L}_n$.

\begin{prop}
\label{prop:ChartSym}
The graph of an element  $u\in \Sym_{P,Q}$ is an element of $U_Q$. This defines an identification of $\Sym_{P,Q}$ with $U_Q\subset \mathcal{L}_n$.

\end{prop}

\begin{proof}

Recall that the graph of a linear map $u:P\to Q$ is the vector subspace of elements $\mathvec{v}+u(\mathvec{v})$ for $\mathvec{v}\in P$. It is a Lagrangian if and only if for all $\mathvec{v},\mathvec{w}\in P$:
$$\omega\left(\mathvec{v}+u(\mathvec{v}),\mathvec{w}+u(\mathvec{w})\right)=0.$$

Since $P,Q$ are Lagrangians this is equivalent to having for all $\mathrm{v},\mathrm{w}\in P$:
$$\omega\left(\mathvec{v},u(\mathvec{w})\right)-\omega\left(\mathvec{w},u(\mathvec{v})\right)=0.$$

Hence the graph of $u$ is a Lagrangian if and only if $u\in \Sym_{P,Q}$.

\end{proof}

\begin{nota}
Let $P,Q$ be transverse Lagrangians and $R$ be a Lagrangian transverse to $Q$, \ie in $U_Q$. We denote by $u^R_{P,Q}$ the corresponding element in $\Sym_{P,Q}$.
\end{nota}

Bilinear symmetric forms can be degenerate: they can have singular spaces. For any vector space $V$ let $\mathcal{Q}(V)$ be the space of symmetric bilinear forms on $V$.

\begin{defn}
\label{defn:SingularSubspace}
A subspace $U$ of a vector space $V$ is \emph{singular} for a symmetric bilinear form $q$ in $\mathcal{Q}(V)$ if for all $\mathvec{v}\in V, \mathvec{w}\in U$, on has $q(\mathvec{v},\mathvec{w})=0$.
\end{defn}

 Let $P,Q$ be two transverse Lagrangians in $\mathcal{L}_n$.

\begin{prop}
\label{prop:ChartQ}
An element $u\in \Sym_{P,Q}$ determines a symmetric bilinear form $q\in \mathcal{Q}(P)$ defined for $\mathrm{v},\mathrm{w}\in P$ as:
$$
q(\mathvec{v},\mathvec{w})=\omega\left(\mathvec{v}, u(\mathvec{w})\right).$$

This defines an identification of $\Sym_{P,Q}$ and $\mathcal{Q}(P)$. Moreover $\Ker(u)$ is singular for $q$.

\end{prop}

%\begin{proof}
%
%Since $P$ and $Q$ are transverse Lagrangians, $\omega$ determines a natural identification between $Q$ and the dual $P^*$ of $P$. Therefore every bilinear form $q$ on $P$ can be written for some unique endomorphism $u:P\to Q$ as $
%q(\mathvec{v},\mathvec{w})=\omega\left(\mathvec{v}, u(\mathvec{w})\right)$ for $\mathrm{v},\mathrm{w}\in P$. This bilinear form is symmetric if and only if $u$ is symmetric.
%
%\medskip
%
%For all $\mathrm{v}\in P$ and $\mathrm{w}\in \Ker(u)$, $q(\mathrm{v},\mathrm{w})=\omega(\mathrm{v},u(\mathrm{w}))=0$, so $\Ker(u)$ is singular for $q$.
%\end{proof}

This identification also defines linear charts $U_Q\simeq \mathcal{Q}(P)$.

\begin{defn}
For $R\in U_Q$, define $q^R_{P,Q}\in \mathcal{Q}(P)$ as the following symmetric bilinear form on $P$:
$$q^R_{P,Q}(\mathvec{v},\mathvec{w})=\omega\left(\mathvec{v}, u^R_{P,Q}(\mathvec{w})\right).$$

\end{defn}

An invariant that classifies orbit of triples of pairwise transverse Lagrangians up to the action of $\Sp(2n,\R)$ is called the Maslov index \cite{Burger2005MaximalRO}. We will be only interested by triples with maximal Maslov index, so we will only define the notion of maximal triples of Lagrangians. For a vector space $V$, let $\mathcal{Q}^+(V)$ denote the open cone of scalar products in the space of symmetric bilinear forms $\mathcal{Q}(V)$.

\begin{defn}
\label{defn:MaximalTriplet}
Let $(P,R,Q)$ be three pairwise transverse Lagrangians in $\R^{2n}$. This triple is called \emph{maximal} if the symmetric bilinear form $q^{R}_{P,Q}$ is in $\mathcal{Q}(P)^+$, \ie is a scalar product.
\end{defn}

A triple $(P,R,Q)$ is maximal in this sense if and only if its Maslov index is maximal, \ie if its Maslov index is equal to $n$ (see for instance \cite{Burger_2017} Lemma 2.10). 

\begin{rem}
\label{rem:MaximalStable}
The signature of $q^R_{P,Q}$ is locally constant on the space of triples of pairwise transverse Lagrangians in $\R^{2n}$. Hence the space of maximal triples of Lagrangians $(P,R,Q)$ forms a connected component of this space. 
\end{rem}

Let us fix an orientation of $\partial \Gamma_g$, \ie a connected component of the space of distinct triples in $\partial \Gamma_g$ that we will call \emph{positive triples}. The \emph{Toledo invariant} of a representation $\rho:\Gamma_g \to \Sp(2n,\R)$ of a surface group $\Gamma_g$ is an integer $T_\rho$ that depends only on the connected component of $\Hom(\Gamma_g,\Sp(2n,\R))$ in which $\rho$ lies. This invariant satisfies $ |T_\rho|\leq n(2g-2)$.  A representation has \emph{maximal} Toledo invariant when $T_\rho=n(2g-2)$ \cite{Burger2005MaximalRO}. The following characterization will be taken as a definition for the rest of the paper.

\begin{defn}
Given an orientation of $\partial \Gamma_g$, we say that a representation $\rho:\Gamma_g\to \Sp(2n,\R)$ is \emph{maximal} if it is $\{n\}$-Anosov and for every positive triple of distinct points $x,y,z$ in $\partial \Gamma_g$, the triple $(x^n_\rho,y^n_\rho,z^n_\rho)$ is maximal, in the sense of Definition \ref{defn:MaximalTriplet}.

\end{defn}

Maximal representations in this sense are exactly representations with maximal Toledo invariant: any representation $\rho$ with maximal Toledo invariant is $\{n\}$-Anosov (\cite{Burger2005MaximalRO}, Theorem 6.1), and its boundary map sends positive triples to maximal triples (\cite{Burger2005MaximalRO}, Theorem 7.6). Conversely any representation that admits a continuous equivariant map from $\partial \Gamma_g$ to $\mathcal{L}_n$ which sends positive triples to maximal triples  has maximal Toledo invariant (\cite{BURGER2003387}, Theorem 8) and in particular is $\lbrace n \rbrace$-Anosov. 

\medskip

An example of the boundary curve $\xi^2_{\rho_0}$ of a maximal representation $\rho_0:\Gamma_g \to \Sp(4,\R)$ is given Figure \ref{fig:Veronese}. The boundary curve which is represented is a part of the Veronese curve, which is the boundary curve of a $4$-Fuchsian representation $\rho_0$, \ie the composition of a fuchsian representation and the irreducible representation $\SL(2,\R)\to \Sp(4,\R)$. The triple $(x,y,z)$ for this picture is a positive triple in $\partial \Gamma_g$. In the picture the point $z^2_{\rho_0}$ is "at infinity".

\begin{figure}
\begin{center}
\tdplotsetmaincoords{80}{45}
\tdplotsetrotatedcoords{-90}{180}{-90}
\begin{tikzpicture}[tdplot_main_coords]

 \coordinate (O) at (0,0,0);
 
 \coneback[surface]{3}{3}{10}
  \conefront[surface]{-3}{-3}{10}
\draw plot[variable=\t,domain=-1.5:1.5,samples=73,smooth] 
(1.73205080757*\t^2,\t-\t^3,\t+\t^3);
 \fill (1.73205080757,0,2) circle (2pt);
 \conefront[surface]{3}{3}{10};

  \node[left] at (1.73205080757,0,2){$y^2_{\rho_0}$};
  \fill (0,0,0) circle (2pt);
  \node[left] at (-0.2,0,0){$x^2_{\rho_0}$};
  \coneback[surface]{-3}{-3}{10};

\end{tikzpicture}
\caption{The Veronese curve $\xi^2_{\rho_0}(\partial \Gamma_g)$ in the chart $\mathcal{Q}(x^2_{\rho_0})\simeq U_{z^2_{\rho_0}}\subset \mathcal{L}_2$ with the cone $\mathcal{Q}^+(x^2_{\rho_0})$.}
\label{fig:Veronese}
\end{center}
\end{figure}
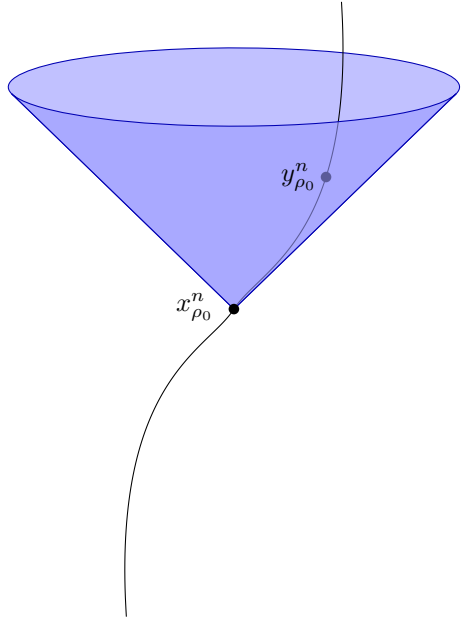

\section{Differentiability properties of the boundary maps.}
\label{sec:Smooth}
The $k$-th boundary map of an Anosov representation $\rho$ of a surface group $\Gamma_g$ has smooth image if $\rho$ satisfies the hyperconvexity property $H_k$, which we now define. Recall that we use Notation \ref{nota:NotationBoundarymaps} for the boundary maps of an Anosov representation.

\begin{defn} 

\label{defn:Hn}

Let $N\geq 2$ and $1\leq k\leq N-1$ be integers. Let $\rho:\Gamma_g\to \SL(N,\R)$ be a $\{k-1,k,k+1\}$-Anosov representation. We say that $\rho$ satisfies property $H_k$ if for all triples of distinct points $x,y,z\in \partial \Gamma_g$, the following sum is direct :
 \begin{equation}
\label{eq:PropHn} 
 \left(x_\rho^k\cap z_\rho^{N+1-k}\right) + \left(y_\rho^{k}\cap z_\rho^{N+1-k}\right) + z_\rho^{N-1-k}.
 \end{equation}
 
 If $\rho$ satisfies property $H_k$ for all $1\leq k\leq N-1$, we say that $\rho$ satisfies property $H$.

\end{defn}

These properties can be also written as follows.

\begin{lem}
\label{rem:equivalentFormHn}
For a triple of distinct points $x,y,z\in \partial \Gamma_g$, the sum \eqref{eq:PropHn} defining property $H_k$ is direct if and only if the following sum is direct:
\begin{equation}
 \label{eq:equivalentFormHn}
 x_\rho^{k} + \left(y_\rho^{k}\cap z_\rho^{N+1-k}\right) + z_\rho^{N-1-k}.
 \end{equation}
\end{lem}

\begin{proof}
The transversality of the boundary maps stated in Proposition \ref{prop:TransversalityBoundaryMaps} implies that the sum 
 $\left(y_\rho^{k}\cap z_\rho^{N+1-k}\right) \oplus z_\rho^{N-1-k}$ is necessarily direct. If a vector in $x_\rho^{k}$ belongs to this sum, it also belongs to $x_\rho^k\cap z_\rho^{N+1-k}$. Hence if \eqref{eq:PropHn} is direct then \eqref{eq:equivalentFormHn} is direct. The converse is immediate since $x_\rho^k\cap z_\rho^{N+1-k}\subset x_\rho^k$.
\end{proof}

 For a $\{k-1,k,k+1\}$-Anosov representation $\rho:\Gamma_g\to\Sp(2n,\R)$, some of these properties are equivalent.
 
\begin{prop}
\label{prop:H_n-k}
Let $\rho:\Gamma_g\to\Sp(2n,\R)$ be a $\{k-1,k,k+1\}$-Anosov representation. It satisfies property $H_k$ if and only if it satisfies property $H_{2n-k}$.
\end{prop}

\begin{proof}
Let $x,y,z\in \partial \Gamma_g$ be distinct points. Let us assume that the sum \eqref{eq:equivalentFormHn} is direct. Hence :
$$x_\rho^{k}\cap \left(\left(y_\rho^{k}\cap z_\rho^{2n+1-k}\right)\oplus z_\rho^{2n-1-k}\right)=\{0\}.$$
By considering the orthogonal of this set with respect to the bilinear form $\omega$, and because of Lemma \ref{lem:ParticularOrthogonality}, one has:
  $$x_\rho^{2n-k}+ \left(\left(y_\rho^{2n-k}\oplus z_\rho^{k-1}\right) \cap z_\rho^{k+1}\right)=\R^{2n}.$$
 
 Since $z_\rho^{k-1}\subset z_\rho^{k+1}$, then $\left(y_\rho^{2n-k}\oplus z_\rho^{k-1}\right) \cap z_\rho^{k+1}=\left(y_\rho^{2n-k}\cap z_\rho^{k+1}\right)\oplus z_\rho^{k-1}$. The following sum is equal to $\R^{2n}$ and the sum of the dimensions on the summands is equal to $2n$, so it is direct: 
 $$x_\rho^{2n-k}+ \left(\left(y_\rho^{2n-k} \cap z_\rho^{k+1}\right)\oplus z_\rho^{k-1}\right)=\R^{2n}.$$
 
This means that this sum is direct for all distinct $x,y,z$, and hence property $H_{2n-k}$ is satisfied. The converse implication is immediate by setting $k'=2n-k$.
\end{proof}

For any $x,y,z\in \partial \Gamma_g$ distinct and any $\{n,n-1\}$-Anosov representation $\rho:\Gamma_g\to \Sp(2n,\R)$, the following subspace is a hyperplane in $x^n_\rho$:
$$(y^{n-1}_\rho\oplus z^n_\rho)\cap x^n_\rho.$$
Indeed $y^{n-1}_\rho\oplus z^n_\rho$ is an hyperplane of $\R^{2n}$ that cannot contain $x^n_\rho$ since $x^n_\rho\oplus  z^n_\rho=\R^{2n}$. This hyperplane can be seen as the image of $y_\rho^{n-1}$ by the linear projection onto $x^n_\rho$ in $\R^{2n}$ associated with the direct sum $\R^{2n}=x^n_\rho\oplus z^n_\rho$.

\medskip

The transversality of boundary maps and property $H_n$ imply the following transversality properties. These properties will be used in the case $n=2$ to prove Lemma \ref{lem:Order4Points} and Theorem \ref{thm:H2toH1}.

\begin{lem}
\label{lem:DistinctPoints}
Let $\rho:\Gamma_g \to \Sp(2n,\mathbb{R})$ be a $\{n-1,n\}$-Anosov representation. Let $x,y,z\in \partial \Gamma_g$ be three distinct points. Then :

\begin{itemize}
\item[(i)] $x_\rho^{n-1}$ and $ z_\rho^{n+1}\cap x_\rho^n$ are transverse;
\item[(ii)]  $x_\rho^{n-1}$ and $ y_\rho^{n+1}\cap x_\rho^n$ are transverse; 
\item[(iii)] $(y_\rho^{n-1}\oplus z^n_\rho)\cap x^n_\rho$ and $z_\rho^{n+1}\cap x_\rho^n$ are transverse;
\item[(iv)] if moreover $\rho$ satisfies property $H_n$, then $(y_\rho^{n-1}\oplus z^n_\rho)\cap x^n_\rho$ and $ y_\rho^{n+1}\cap x_\rho^n$ are transverse.
\end{itemize}
\end{lem}

\begin{proof}
The transversality of the boundary maps between $x$ and $z$ implies that  $x_\rho^{n-1}$ and $z_\rho^{n+1}$ have trivial intersection so $x_\rho^{n-1}$ and $ z_\rho^{n+1}\cap x_\rho^{n-1}$ intersect trivially. The same argument shows that, $x^{n-1}_\rho$ and $ y_\rho^{n+1}\cap x_\rho^n$ are disjoint.

\medskip

The transversality of the boundary maps between $y$ and $z$ implies that $y_\rho^{n-1}$ and $z_\rho^{n+1}$ have trivial intersection. In particular let $\mathvec{v}\in  (y_\rho^{n-1}\oplus z^n_\rho)\cap x^n_\rho$ and $\mathvec{w}\in z^n_\rho$ be such that $\mathvec{v}+\mathvec{w}\in y_\rho^{n-1}$. Suppose that moreover $\mathvec{v}\in z_\rho^{n+1}$. Then $\mathvec{v}+\mathvec{w}\in y_\rho^{n-1} \cap z_\rho^{n+1}$ since $z^n_\rho\subset z_\rho^{n+1}$. Hence $\mathvec{v}+\mathvec{w}=0$, so $\mathrm{v}\in x^n_\rho\cap z^n_\rho$. Therefore $\mathrm{v}=0$. As a conclusion $(y_\rho^{n-1}\oplus z^n_\rho)\cap x^n_\rho$ and $ z_\rho^{n+1}\cap x_\rho^n$ are disjoint.
 
\medskip

Finally property $H_n$ implies that if we replace $(x,y,z)$ by $(z,x,y)$ in \eqref{eq:equivalentFormHn}, the sum is direct, and hence $x_\rho^n\cap y_\rho^{n+1}$ intersects trivially $z^n_\rho\oplus y_\rho^{n-1}$. Therefore $(y_\rho^{n-1}\oplus z^n_\rho)\cap x^n_\rho$ and $ y_\rho^{n+1}\cap x_\rho^n$ are disjoint.
\end{proof}

The main tool that we are going to use in Sections \ref{sec:MaximalityHn} and \ref{sec:H2toH1} is the following result from Pozzetti, Sambarino and Wienhard \cite{pozzetti2020conformality}.

\begin{thm}[{\cite{pozzetti2020conformality}, Theorem 4.2}]
\label{thm:SmoothBoundary}
Let $\rho:\Gamma_g \to \SL(N,\mathbb{R})$ be a $\{k-1,k,k+1\}$-Anosov representation. If $\rho$ satisfies property $H_k$ then the map $\xi_\rho^k: x\mapsto x^k_\rho$ has $\mathcal{C}^1$ image, \ie $\xi^k_\rho(\partial \Gamma_g)\subset \Grass(k,N)$ is a $1$-dimensional $\mathcal{C}^1$ submanifold.

At the point $x_\rho^k$ this $1$-dimensional submanifold of $\Grass(k,N)$ is tangent to the curve consisting of spaces containing $x_\rho^{k-1}$ and contained in $x_\rho^{k+1}$.
\end{thm}

We will be interested in the regularity of the boundary curve $\xi^n_\rho$, whose image lies in the space of Lagrangians $\mathcal{L}_n$ when $\rho(\Gamma_g)\subset \Sp(2n,\R)$. Once an $\{n\}$-Anosov representation $\rho$ has been fixed, given $3$ points $x,y,z\in \partial \Gamma_g$ with $x,y\neq z$ we will write for simplicity :
$$\Sym_{x,z}:=\Sym_{x_\rho^n,z_\rho^{n}}\;\;,\;\; u_{x,z}^y:=u_{x_\rho^n,z_\rho^{n}}^{y^n_\rho}\in \Sym_{x,y}\;\;,\;\;q^y_{x,z}=q_{x_\rho^n,z_\rho^{n}}^{y^n_\rho}\in \mathcal{Q}(x^n_\rho).$$

 Let us rephrase Theorem \ref{thm:SmoothBoundary} in the charts $U_Q\simeq \Sym_{x,z}$ of the space of Lagrangians $\mathcal{L}_n$.

\begin{lem}
\label{lem:DescriptionDerivativeBondaryMap1}
Let $\rho:\Gamma_g\to \Sp(2n,\R)$ be an $\{n-1,n\}$-Anosov representation that satisfies property $H_{n}$. Let $x,z\in \partial \Gamma_g$ be distinct points. 
 For $y\neq z$, the tangent space at $u^y_{x,z}$ to the image of the map:
 $$w\in \partial\Gamma_g\setminus\{z\}\mapsto u^w_{x,z}$$
 is the affine line of $\Sym_{x,z}$ passing through $u^y_{x,z}$ and directed by the vector line of elements $\dot{u}\in \Sym_{x,y}$ such that one of the following equivalent statements holds:
\begin{itemize}
\item[(i)] $\left(y^{n-1}_\rho\oplus z^n_\rho\right)\cap x^n_\rho\subset\Ker(\dot{u}),$
\item[(ii)] $\im(\dot{u})\subset y^{n+1}_\rho\cap z^n_\rho.$
\end{itemize}
In particular such an element $\dot{u}\in \Sym_{x,z}$ must have rank $1$. 
\end{lem}

\begin{proof}
Because of Theorem \ref{thm:SmoothBoundary}, the image of the boundary map $\xi^n_\rho$ is $\mathcal{C}^1$, and tangent at $y^n_\rho$ to the one dimensional submanifold $\ell$ of $\mathcal{L}_n$ consisting of Lagrangians $P$ satisfying the condition :
 $$y^{n-1}_\rho\subset P\subset y^{n+1}_\rho.$$
Since $y^{n-1}_\rho$ is orthogonal to $y^{n+1}_\rho$ with respect to $\omega$, and since $P$ is a Lagrangian, this is equivalent to $y^{n-1}_\rho\subset P$ which is equivalent to $P\subset y^{n+1}_\rho$.

\medskip

An element $u'\in \Sym_{x,z}$ corresponds to a Lagrangian $P$ satisfying $y^{n-1}_\rho\subset P$ if and only if for all $\mathvec{v}\in x^n_\rho$ such that $\mathvec{v}+u_{x,z}^y(\mathvec{v})\in y^{n-1}_\rho$ one has $\mathvec{w}+u'(\mathvec{w})=\mathvec{v}+u_{x,z}^y(\mathvec{v})$ for some $\mathrm{w}\in x^n_\rho$ that must be equal to $\mathrm{v}$ since $\mathrm{v}-\mathrm{w}\in x^n_\rho\cap z^n_\rho$. But $\mathrm{v}+u_{x,z}^y(\mathrm{v})\in y^{n-1}_\rho$ if and only if $\mathrm{v}\in (y^{n-1}_\rho\oplus z^n_\rho)\cap x^n_\rho$. Hence $y^{n-1}_\rho\subset P$ if and only if :
$$\left(y^{n-1}_\rho\oplus z^n_\rho\right)\cap x^n_\rho\subset \Ker(u'-u_{x,z}^y).$$

Similarly an element $u'\in \Sym_{x,z}$ corresponds to a Lagrangian $P$ satisfying  $P\subset y^{n+1}$ if and only if for all $\mathvec{v}\in x_\rho^{n}$, $\mathvec{v}+u'(\mathvec{v})\in y^{n+1}_\rho $. However $\mathvec{v}+u_{x,z}^y(\mathvec{v})\in y^n_\rho\subset y^{n+1}_\rho$. Hence $P\subset y^{n+1}_\rho $ if and only if for all $\mathrm{v}\in x^n_\rho$, $u'(\mathrm{v})-u^y_{x,z}(\mathrm{v})\in y^{n+1}_\rho$, or in other words:
$$\im(u'-u_{x,z}^y)\subset y^{n+1}_\rho \cap z^n_\rho.$$

\medskip

Therefore the image $\ell'$ of the submanifold $\ell$ in the chart $\Sym_{P,Q}$ is the affine line directed by symmetric endomorphisms $\dot{u}$ satisfying $\left(y^{n-1}_\rho\oplus z^n_\rho\right)\cap x^n_\rho\subset\Ker(\dot{u})$, or equivalently $\im(\dot{u})\subset y^{n+1}_\rho \cap z^n_\rho$. Such a non-zero element must have rank $1$.

\end{proof}

Theorem \ref{thm:SmoothBoundary} can be also rephrased in the chart $U_Q\simeq \mathcal{Q}(x^n_\rho)$ of $\mathcal{L}_n$. Recall that singular subspaces for a symmetic bilinear form were defined in Definition \ref{defn:SingularSubspace}.

\begin{lem}
\label{lem:DescriptionDerivativeBondaryMap2}
Let $\rho:\Gamma_g\to \Sp(2n,\R)$ be a $\{n-1,n\}$-Anosov representation that satisfies property $H_{n}$. Let $x,z\in \partial \Gamma_g$ be distinct points.
For $y\neq z$, the tangent space at $q^y_{x,z}$ to the image of the map 
$$w\in \partial\Gamma_g\setminus\{z\}\mapsto q^w_{x,z}$$ 
 is the affine line of $\mathcal{Q}(x^n_\rho)$ passing through $q^y_{x,z}$ and directed by the vector line of elements $\dot{q}\in \Sym_{x,z}$ such that the hyperplane $\left(y^{n-1}_\rho\oplus z^n_\rho\right)\cap x^n_\rho$ is singular for $\dot{q}$.
In particular such an element $\dot{q}\in \mathcal{Q}(x^n_\rho)$ must have signature $(1,0)$ or $(0,1)$.  
\end{lem}

\begin{proof}
Let $\ell'$ be the affine line in $\Sym_{x,z}$ defined in the proof of Lemma \ref{lem:DescriptionDerivativeBondaryMap1} part (i). 
The  affine line $\tilde{\ell}$ in $\mathcal{Q}(x^n_\rho)$ corresponding to $\ell'$ via the linear identification $\Sym_{x,z}\simeq \mathcal{Q}(x^n_\rho)$ is directed by the elements $\dot{q}\in\mathcal{Q}(x^n_\rho)$ such that for some $\dot{u}\in \Sym_{P,Q}$ satisfying (i), and for all $\mathrm{v}\in  x^n_\rho$ and $\mathrm{w}\in (y^{n-1}_\rho\oplus z^n_\rho)\cap x^n_\rho$, one has $\dot{q}(\mathrm{v},\mathrm{w})=\omega(\mathrm{v},\dot{u}(\mathrm{w}))=0$.

 In other words $\tilde{\ell}$ is directed by the non-zero elements $\dot{q}\in \mathcal{Q}(x^n_\rho)$ such that $\dot{q}(\mathrm{v},\mathrm{w})=0$ for all $\mathrm{v}\in z^n_\rho$ and $\mathrm{w}\in\left(y^{n-1}_\rho\oplus x^n_\rho\right)\cap z^n_\rho$, \ie such that $\left(y^{n-1}_\rho\oplus z^n_\rho\right)\cap x^n_\rho$ is singular for $\dot{q}$. 

Since $\dot{q}$ is non-zero but admits a singular hyperplane, its signature is equal to $(1,0)$ or $(0,1)$
\end{proof}

A first application of this result is the following lemma, which will be used in the proof of Lemma \ref{lem:IntermediateH1}.

\begin{lem}
\label{lem:NonConstant}

Let $\rho$ be a $\{n-1,n\}$-Anosov representation that satisfies property $H_n$. Let $z\in \partial \Gamma_g$. 
The map that associates to $y\in \partial \Gamma_g\setminus \{z\}$ the hyperplane $y^{n-1}_\rho\oplus z^n\subset \R^{2n}$ is constant on no open interval.
\end{lem}

\begin{proof}
Let $x\in \partial \Gamma_g$ be any point distinct from $z$. Let $\psi$ be the map that associates to an element $y\in \partial \Gamma_g\setminus \{z\}$ the following hyperplane of $x^n_\rho$:
$$\psi(y)=\left(y^{n-1}_\rho\oplus z^n_\rho\right)\cap x^n_\rho.$$

If this map was constant on some open interval $I$, Lemma \ref{lem:DescriptionDerivativeBondaryMap1} would imply that the image by $y\mapsto u^y_{x,z}\in \Sym_{x,z}$ restricted to $I$ has a tangent direction which is always a rank one symmetric element with constant kernel $\psi(x)$. 

\medskip

However in this case $u_{x,z}^y$ would be the integral of  some elements $\dot{u}\in \Sym_{x,z}$ whose kernel always contains $\psi(x)$. In particular $u^y_{x,z}$ would have rank at most $1$. However this would imply that this element has a kernel, and hence $y^n_\rho$ has a non-trivial intersection with $x^n_\rho$. This would contradict the transversality of the boundary maps (Proposition \ref{prop:TransversalityBoundaryMaps}). 

\medskip

Hence the map $\psi$ cannot be constant on any open interval.
\end{proof}

\section{Relation between maximality and property \texorpdfstring{$H_n$}{Hn}.}
\label{sec:MaximalityHn}

Our goal will be to prove that a $\{n-1,n\}$-Anosov representation $\rho$ is maximal if and only if it satisfies property $H_n$. In order to prove property $H_n$ implies maximality we will use the smoothness of the $n$-th boundary curve and the following simple geometric fact. 

\medskip

A \emph{closed cone} of a vector space is a closed subset that is stable by addition and multiplication by positive scalars.

\begin{lem}
\label{lem:ConeConvex}

 Let $V$ be a real vector space and $S$ be a closed cone in $V$, Let $\eta :\mathbb{R}\to V$ be a $\mathcal{C}^1$ curve such that for all $t\in \mathbb{R}$, $\eta'(t)\in S$ and $\eta(0)\in S$.
Then, for all $t\geq 0$, $\eta(t)\in S$.

\end{lem}

In other words, if the derivative of a curve stays in a closed cone and if the curve is in the closed cone initially, then the curve stays in this closed cone. We will use this fact again to prove Lemma \ref{lem:Order4Points}.

\begin{proof}

Let $t\geq 0$ be a real number, the we can write $\eta(t)$ as :
$$\eta(t)=\eta(0)+\int_{0}^t \eta'(s)\mathrm{d}s.$$

Hence $\eta(t)$ can be approximated by finite sums of elements in $S$, which are also in $S$ since $S$ is a cone. Moreover $S$ is closed so $\eta(t)\in S$.
\end{proof}

Now we prove the following characterization of maximal representations that are $\{n-1,n\}$-Anosov. 

\begin{thm}
\label{thm:MaximalH_n}
Let $1\leq k\leq n$. Let $\rho:\Gamma_g\to\Sp(2n,\R)$ be a $\{n-1,n\}$-Anosov representation. 
The representation $\rho$ satisfies property $H_{n}$ if and only if it is maximal for some orientation of $\partial \Gamma_g$.
\end{thm}

\begin{proof}

Suppose first that $\rho$ is maximal for some orientation of $\Gamma_g$. Let $(x,y,z)$ be a positive triple of distinct points in $\partial \Gamma_g$. Suppose that the sum \eqref{eq:PropHn} is not direct, \ie that there is a vector $\mathvec{h}$ belonging to the intersection:
$$\left(\left(x_\rho^{n}\cap z_\rho^{n+1}\right) \oplus  z_\rho^{n-1}\right)\cap\left(y_\rho^{n}\cap z_\rho^{n+1}\right).$$

 Note that in this expression, $\left(x_\rho^{n}\cap z_\rho^{n+1}\right) \oplus  z_\rho^{n-1}$ is direct since $x_\rho^n\cap z_\rho^{n-1}=\lbrace 0\rbrace$. In particular $\mathvec{h}=\mathvec{v}+\mathvec{w}$ for some $\mathvec{v}\in x_\rho^{n}\cap z_\rho^{n+1}$, $\mathvec{w}\in z_\rho^{n-1}$.
Moreover $\mathvec{h}\in y_\rho^{n}$ so $u^y_{x,z}(\mathvec{v})=\mathvec{w}$ with $u^y_{x,z}\in \Sym_{x,z}$ the element corresponding to $y^n_\rho$. 

\smallskip

Lemma \ref{lem:ParticularOrthogonality} implies that $z_\rho^{n+1}$ and $z_\rho^{n-1}$ are orthogonal with respect to $\omega$, so $\omega(\mathvec{v},\mathvec{w})=0$. Thus the symmetric bilinear form $q^y_{x,z}$ associated to $u^y_{x,z}$ satisfies $q^y_{x,z}(\mathvec{v},\mathvec{v})=\omega\left(\mathvec{v},u^y_{x,z}(\mathvec{v})\right)=0$. However $q^y_{x,z}$ is positive since $\rho$ is maximal and $(x,y,z)$ is positive. Hence $\mathvec{v}=\mathvec{w}=0$ and the desired sum of spaces is direct. This means that the sum is direct for all positive triples $(x,y,z)$, but since \eqref{eq:PropHn} stays invariant when $x$ and $y$ are exchanged, this sum is direct for all triples. Therefore if $\rho$ is maximal, then property $H_{n}$ holds.

\bigskip

Conversely, let us suppose that $\rho$ satisfies $H_n$. Let $x,z\in \partial \Gamma_g$ be distinct points. Lemma \ref{lem:DescriptionDerivativeBondaryMap2} implies that there exists a parametrization $\phi:\R\to \partial \Gamma_g\setminus \{z\}$ which is a homeomorphism such that $f:\R\to \mathcal{Q}(x^n_\rho)$, $t\mapsto q^{\phi(t)}_{x,z}$ is a $\mathcal{C}^1$ embedding. 

The derivative of $f$ at all times is non-zero and has signature $(1,0)$ or $(0,1)$. Up to considering $t\mapsto \phi(-t)$, we can assume that at some point the derivative of $f$ has signature $(1,0)$, and hence it has signature $(1,0)$ for all points $t\in \R$. 

\smallskip

Let us assume that $\phi(0)=x$. The derivative of $t\mapsto q^{\phi(t)}_{x,z}$ has signature $(1,0)$, hence it belongs to the closed cone $\overline{\mathcal{Q}^+(x^n_\rho)}$ of semi-positive elements. Moreover $q^x_{x,z}=0$ is also in this closed cone. Hence by Lemma \ref{lem:ConeConvex} the image of $\mathbb{R}_{\geq 0}$ consists only of semi-positive elements. 

For $t> 0$, the transversality of boundary maps implies that:
 $$x^n_\rho\cap\xi^n_\rho\circ \phi(t)=\{0\}.$$
 
 Hence $q^{\phi(t)}_{x,z}$ is a non-degenerate symmetric bilinear form. Since it belongs to $\overline{\mathcal{Q}^+(x^n_\rho)}$, it is positive. Therefore the triple of Lagrangians $(x^n_\rho,\xi^n_\rho\circ \phi(t),z^n_\rho)$ is maximal for all $t>0$.

\smallskip

Hence for at least one triple of distinct points $x,y,z\in \partial \Gamma_g$ , the triple $(x^n_\rho,y^n_\rho,z^n_\rho)$ is maximal. Because of Remark \ref{rem:MaximalStable}, this holds for all triple of distinct points $(x',y',z')$ ordered as $(x,y,z)$ in $\partial \Gamma_g$. Therefore for the orientation of $\partial \Gamma_g$ such that $(x,y,z)$ is positive, the representation $\rho$ is maximal.
\end{proof}

\section{From property \texorpdfstring{$H_2$}{H2} to \texorpdfstring{$H_1$}{H1} for \texorpdfstring{$\Sp(4,\R)$}{Sp(4,R)}.}
\label{sec:H2toH1}

We proved in Section \ref{sec:MaximalityHn} that any maximal and $\{n-1,n\}$-Anosov representation satisfies property $H_n$. This means that we can use Theorem \ref{thm:SmoothBoundary} to get more information on the boundary curve.
In this section we prove that if additionally $n=2$, such a representation must also satisfy property $H_1$. 
 
\medskip

For a triple $(x,y,z)$ of distinct points in the circle $\partial\Gamma_g$, let $(x,y)_z$ and $[x,y]_z$ be respectively the open and closed arc in the circle $\partial \Gamma_g$ between $x$ and $y$ not containing $z$.

Before we prove the key result of this section in Lemma \ref{lem:Order4Points}, we need to find a positive triple of points in $\partial \Gamma_g$ that satisfies the following lemma. Given a triple $(x,w,z)$ in $\partial \Gamma_g$ such that $x,w\neq z$  we define:
$$\psi(w)=\left(w_\rho^1\oplus z^2_\rho\right)\cap x^2_\rho\in \mathbb{P}(x^2_\rho).$$

\begin{lem}
\label{lem:IntermediateH1}
Let $\rho:\Gamma_g\to \Sp(4,\R)$ be a $\{1,2\}$-Anosov representation that is maximal for some orientation of $\partial \Gamma_g$. There exists a positive triple $(x,y,z)$ in $\partial \Gamma_g$ such that $\psi(x)\neq \psi(y)$ and  for all $w\in (x,y)_z$ and $\psi(w)\neq \psi(x),\psi(y)$.
\end{lem}

\begin{proof}

Let $z\in \partial \Gamma_g$ be any point. Let $\psi_0$ be the map that associates to an element $w\in \partial \Gamma_g\setminus \{z\}$ the hyperplane $w_\rho^1\oplus z^2_\rho\subset \R^4$. Theorem \ref{thm:MaximalH_n} implies that $\rho$ satisfies property $H_n$, and because of Lemma \ref{lem:NonConstant} the map $\psi_0$ is not constant. 

\medskip

In particular we can find some distinct $x_0,y_0\in \partial \Gamma_g\setminus\{z\}$ such that $\psi_0(x_0)\neq \psi_0(y_0)$.
Let $x\in [x_0,y_0]_z$ be the unique point such that $\psi_0(x)=\psi_0(x_0)$ and for all $w\in (x,y_0)_z$, $\psi_0(w)\neq \psi_0(x)$.
Then define similarly $y\in [x,y_0]_z$ as the unique point such that $\psi_0(y)=\psi_0(y_0)$ and for all $w\in (x,y)_z$, $\psi_0(w)\neq \psi_0(y)$.

Hence $\psi_0(x)\neq \psi_0(y)$, and for all $w\in (x,y)_z$, $\psi_0(w)\neq \psi_0(x),\psi_0(y)$. Up to exchanging $x$ and $y$ we can assume that $(x,y,z)$ is a positive triple.

\medskip

Now that we fixed a triple $(x,y,z)$, the map $\psi$ is defined. For all $w\neq z$, $\psi_0(w)=\psi(w)\oplus z^2_\rho$. Hence $\psi(x)\neq \psi(y)$ and for all $w\in (x,y)_z$,  $\psi(w)\neq \psi(x), \psi(y)$. The triple $(x,y,z)$ satisfies the desired condition.

\end{proof}
 
Let $P,Q$ be two Lagrangians in $\R^4$. The space $\mathbb{P}(\mathcal{Q}^+(P))$ of positive symmetric bilinear forms on $P$ up to a scalar is a projective model of the hyperbolic plane $\mathbb{H}^2$.
There is a natural identification $\iota:\mathbb{P}(P)\to \mathbb{P}(\partial \mathcal{Q}^+(P))$.
To a line $\ell\in \mathbb{P}(P)$ we can associate the line $\iota(\ell)$ of symmetric bilinear elements elements $q\in \mathcal{Q}(P)$  for which $\ell$ is singular (see Definition \ref{defn:SingularSubspace}).

\medskip

Recall that we use Notation \ref{nota:NotationBoundarymaps} for the boundary maps of an Anosov representation.
Given a $\{1,2\}$-Anosov representation $\rho$ that satisfies property $H_2$, the fact that $y^1_\rho\subset y^2_\rho\subset y^3_\rho$ implies the following result.

\begin{lem}
\label{lem:Aligment3ptsQ}
Let $\rho:\Gamma_g\to \Sp(4,\R)$ be a $\{1,2\}$-Anosov representation that satisfies property $H_2$.
Let $(x,y,z)\in \partial \Gamma_g$ be distinct points. The point $\left[q^y_{x,z}\right]\in \mathbb{P}(\mathcal{Q}^+(x^2_\rho))$ lies in the projective line between the two elements of $\mathbb{P}(\partial \mathcal{Q}^+(x^2_\rho))$:
$$\iota(y^3_\rho\cap x^2_\rho)\;,\;\iota\left((y^1_\rho\oplus z^2_\rho)\cap x^2_\rho\right).$$
\end{lem}

This projective line is illustrated as a dotted line in Figure \ref{fig:IllustrationCircles}.
Through the identification $\mathbb{P}(\mathcal{Q}^+(x^2_\rho))\cong \mathbb{H}^2$, this line corresponds to a geodesic.

\begin{proof}
Let $\mathrm{v}\in y^3_\rho\cap x^2_\rho$ be a non-zero vector. One has $u^y_{x,z}(\mathrm{v})+\mathrm{v}\in y^2_\rho\subset y^3_\rho$ and hence $u^y_{x,z}(\mathrm{v})\in y^3_\rho\cap z^2_\rho$.

Let $u_0\in \Sym_{x,y}$ be such that $\Ker(u_0)=(y^1_\rho\oplus z^2_\rho)\cap x^2_\rho$.
Since $u_0$ is symmetric, $\im(u_0)$ is orthogonal with respect to $\omega$ to $\Ker(u_0)$.
Since $y^1_\rho$ and $y^3_\rho$ are orthogonal with respect to $\omega$, then $\im(u_0)=y^3_\rho\cap z^2_\rho$.
Hence $u_0(\mathrm{v})\in y^3_\rho\cap z^2_\rho$, therefore $u_0(\mathrm{v})$ and $u^y_{x,z}(\mathrm{v})$ are collinear.

\medskip

By the part (iv) of Lemma \ref{lem:DistinctPoints} and since $\rho$ satisfies property $H_2$, $y^3_\rho\cap x^2_\rho\cap \Ker(u_0)=\{0\}$, and hence $u_0(\mathrm{v})\neq 0$.
Therefore, for some $\lambda\in \R$, $u_1(\mathrm{v})=0$ with $u_1=u^{y}_{x,z}-\lambda u_0$.
In particular $q^y_{x,z}=q_1+\lambda q_0$ where $q_0, q_1\in \mathcal{Q}(P)$ are such that $(y^1_\rho\oplus z^2_\rho)\cap x^2_\rho$ is singular for $q_0$ and $y^3_\rho\cap x^2_\rho$ is singular for $q_1$.
Hence $q_1\in \iota(y^3_\rho\cap x^2_\rho)$ and $q_2\in \iota\left((y^1_\rho\oplus z^2_\rho)\cap x^2_\rho\right)$, which concludes the proof.

\end{proof}

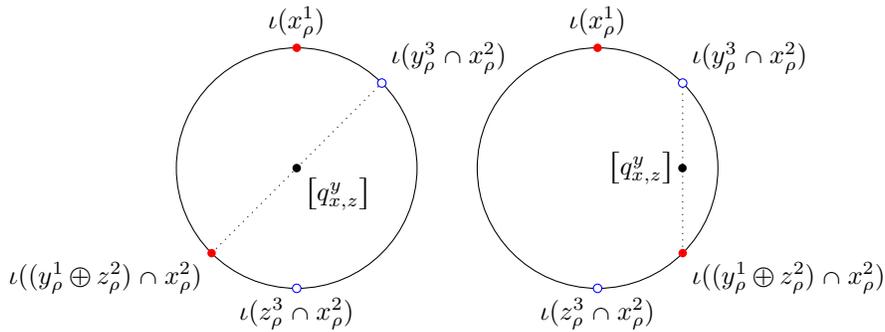
\begin{figure}[!ht]
\begin{center}
\begin{tikzpicture}[scale=0.8]

\draw (0,0) circle (2);

\fill (0,0) circle (2pt);
\node[below right] at (0,0) {$\left[q_{x,z}^y\right]$};
\fill[red] (90:2) circle (2pt);
\node[above] at (90:2) {$\iota(x^1_\rho)$};
\draw[blue,fill=white] (270:2) circle (2pt);
\node[below] at (270:2) {$\iota(z_\rho^3\cap  x_\rho^2)$};

\fill[red] (225:2) circle (2pt);
\node[above right] at (45:2) {$\iota(y_\rho^3\cap  x_\rho^2)$};
\node[below left] at (225:2) {$\iota((y_\rho^1\oplus z^2_\rho)\cap x^2_\rho)$};
\draw[dotted] (225:2) -- (45:2);
%\draw[->] (0,0) -- (45:0.5);
\draw[blue,fill=white] (45:2) circle (2pt);

\begin{scope}[shift={(5,0)}]

\draw (0,0) circle (2);
\fill (1.41,0) circle (2pt);
\node[left] at (1.41,0) {$\left[q_{x,z}^y\right]$};
\fill[red] (90:2) circle (2pt);
\node[above] at (90:2) {$\iota(x^1_\rho)$};
\draw[blue,fill=white] (270:2) circle (2pt);
\node[below] at (270:2) {$\iota(z_\rho^3\cap  x_\rho^2)$};
\draw[blue,fill=white] (45:2) circle (2pt);
\fill[red] (-45:2) circle (2pt);

\node[below right] at (-45:2) {$\iota((y_\rho^1\oplus z^2_\rho)\cap x^2_\rho)$};
\draw[dotted] (-45:2) --(45:2);
\node[above right] at (45:2) {$\iota(y_\rho^3\cap  x_\rho^2)$};
%\draw[->] (1.41,0) -- +(90:0.5);
\end{scope}
\end{tikzpicture}
\end{center}
\caption{Two possible configurations of the image by $\iota$ of the points in \eqref{eq:OrderPoints} in $\mathbb{P}(\mathcal{Q}(x^2_\rho))$.}
\label{fig:IllustrationCircles}
\end{figure}

In order to state Lemma \ref{lem:Order4Points}, we need to define a notion of \emph{cyclically oriented} quadruple on a topological circle.

\begin{defn}

\label{defn:Cyclically ordered}
Let $V$ be a vector space of dimension $2$. A quadruple $(a,b,c,d)$ of points in $\mathbb{P}(V)$ is \emph{cyclically ordered} if $b$ and $d$ are in different components of $\mathbb{P}(V)\setminus \{a,c\}$, or equivalently if the following cross ratio is negative:
\begin{equation}\label{eq:CrossRatio}
\mathrm{cr}(a,b;c,d):=\frac{\overline{a}\wedge \overline{b}}{\overline{c}\wedge \overline{b}} \times\frac{\overline{c}\wedge \overline{d}}{\overline{a}\wedge \overline{d}}.
\end{equation}

Here $\overline{a},\overline{b},\overline{c},\overline{d}$ are any non-zero vectors representing the lines $a,b,c,d$.
\end{defn}

The key argument of the proof of Theorem \ref{thm:H2toH1} is the Lemma \ref{lem:Order4Points}. We will use the geometric fact from Lemma \ref{lem:ConeConvex} that a curve whose derivative lies in a cone must remain in that cone.

\begin{lem}
\label{lem:Order4Points}
Let $\rho:\Gamma_g\to \Sp(4,\R)$ be a $\{1,2\}$-Anosov representation that satisfies property $H_2$. 
There exists some triple of distinct points  $x,y,z\in \partial \Gamma_g$ such that the quadruple
\begin{equation}
 \label{eq:OrderPoints}
(\;z_\rho^3\cap  x_\rho^2,\;\;(y_\rho^1\oplus z^2_\rho)\cap x^2_\rho,\;\;y_\rho^3\cap  x_\rho^2,\;\;x^1_\rho\;)
\end{equation}

is cyclically ordered in $\mathbb{P}(x_\rho^2)$.
\end{lem}
Figure \ref{fig:IllustrationCircles} illustrates this Lemma. The depicted filled points are distinct from the unfilled ones because of Lemma \ref{lem:DistinctPoints}. The $4$ points depicted are cyclically ordered as in \eqref{eq:OrderPoints} on the right picture, but not on the left.

\begin{proof}
Our goal is to find $x,y,z\in\partial \Gamma_g$ such that the $4$ points in \eqref{eq:OrderPoints} are cyclically ordered, \ie are not ordered as in the left part of Figure \ref{fig:IllustrationCircles}. Because of Theorem \ref{thm:MaximalH_n}, we can choose an orientation of $\partial \Gamma_g$ such that $\rho$ is maximal. Let $(x,y,z)$ be a positive triple in $\partial \Gamma_g$ that satisfies the properties from Lemma \ref{lem:IntermediateH1}. Let $\psi(w)=\left(w_\rho^1\oplus z^2_\rho\right)\cap x^2_\rho$ for $w\neq z$ as defined in Lemma \ref{lem:IntermediateH1}. Note that $\psi(x)=x^1_\rho$.

\medskip

Let us assume that the four points in \eqref{eq:OrderPoints} are not cyclically ordered for the triple $(x,y,z)$. This means that $y^3_\rho\cap x^2_\rho$ and $z^3_\rho\cap x^2_\rho$ are in the same connected component of $\mathbb{P}(x^2_\rho)\setminus\{\psi(x),\psi(y)\}$. 

\medskip

The linear plane in $\mathcal{Q}(x^2_\rho)$ passing through $\iota(\psi(x))$ and $\iota(\psi(y))$ cuts the closure $\overline{\mathcal{Q}^+(x^2_\rho)}$ of the cone of scalar products into two closed cones. Let $C$ be the closure of the one whose projectivization satisfies $\iota(y^3_\rho\cap x^2_\rho),\iota(z^3_\rho\cap x^2_\rho)\notin \mathbb{P}(C)$.

The convex set $\mathbb{P}(C)$ is illustrated as the colored region in Figure \ref{fig:IllustrationConvex}. One has $\iota(\psi(x)),\iota(\psi(y))\in\partial\mathbb{P}(C)$. Moreover Lemma \ref{lem:Aligment3ptsQ} implies that $[q^y_{x,z}]$ lies in the segment between $\iota(\psi(y))$ and $\iota(y^3_\rho\cap x^2_\rho)$. As a consequence $[q^y_{x,z}]\notin \mathbb{P}(C)$.

\medskip

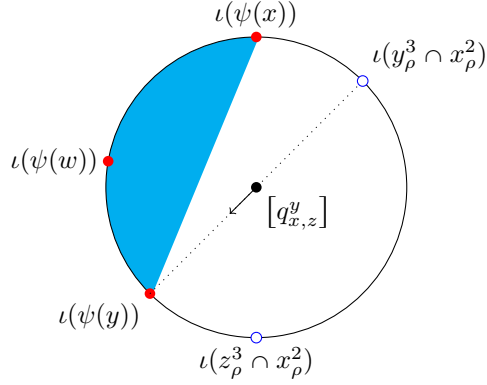
\begin{figure}[!ht]
\usetikzlibrary{patterns}
\begin{center}
\begin{tikzpicture}
\fill[cyan] (89:2) arc (89:226:2) -- cycle;
\draw (0,0) circle (2);

\fill (0,0) circle (2pt);
\node[below right] at (0,0) {$\left[q_{x,z}^y\right]$};
\fill[red] (90:2) circle (2pt);
\node[above] at (90:2) {$\iota(\psi(x))$};
\draw[blue,fill=white] (270:2) circle (2pt);
\node[below] at (270:2) {$\iota(z^3_\rho\cap x^2_\rho)$};
\fill[red] (225:2) circle (2pt);
\node[above right] at (45:2) {$\iota(y^3_\rho\cap x^2_\rho)$};
\fill[red] (170:2) circle (2pt);
\node[left] at (170:2){$\iota(\psi(w))$};
\node[below left] at (225:2) {$\iota(\psi(y))$};

\draw[dotted] (225:2) -- (45:2);
\draw[->] (0,0) -- (225:0.5);
\draw[blue,fill=white] (45:2) circle (2pt);
\end{tikzpicture}
\end{center}
\caption{$\mathbb{P}(\mathcal{Q}(x^2_\rho))$ and the convex $\mathbb{P}(C)$ from the proof of Lemma \ref{lem:Order4Points}.}
\label{fig:IllustrationConvex}
\end{figure}

Because of Lemma \ref{lem:DescriptionDerivativeBondaryMap2}, there exists a parametrization $\phi:\mathbb{R}\to \partial \Gamma_g\setminus \{z\}$, such that $\phi(0)=x, \phi(1)=y$ and $\xi^2_\rho\circ \phi$ is a $\mathcal{C}^1$ embedding. Let $\dot{q}(t_0)$ be the derivative at $t=t_0$ of the map:
$$t\mapsto q^{\phi(t)}_{x,z}.$$

Since $(x,y,z)$ is positive, for any $t_0\in \R$, $\dot{q}(t_0)$ is an element of $\overline{\mathcal{Q}^+(x^2_\rho)}$, whose projectivization is $\iota(\psi(\phi(t_0)))$. 

\medskip

The map $g:t\mapsto \iota(\psi(\phi(t)))$ is continuous from $[0,1]$ to the circle $\mathbb{P}(x^2_\rho)$.
Because $(x,y,z)$ was chosen as in Lemma \ref{lem:Aligment3ptsQ}, one has $g(0)\neq g(1)$ and for $t\in (0,1)$, $g(t)\neq g(0),g(1)$.
Hence $g([0,1])$ is equal to one of the two arcs joining $g(0)$ and $g(1)$.
Because of Lemma \ref{lem:DistinctPoints}, $\iota(z^3_\rho\cap x^2_\rho)$ is not in the image of this map. 

Therefore $g([0,1])=\iota(\psi([x,y]_z))$ is equal to the closed arc in $\iota(\mathbb{P}(x^2_\rho))$ between $\iota(\psi(x))$ and $\iota(\psi(y))$ not containing $\iota(z^3_\rho\cap x^2_\rho)$.
In particular $\iota(\psi(w))\in \mathbb{P}(C)$ for $w\in [x,y]_z$.
Hence for all $t\in [0,1]$, $\dot{q}(t)\in C$. 

\medskip

Moreover $0=q^x_{x,z}=q^{\phi(0)}_{x,z}\in C$.
Hence, since $C$ is a closed cone, by Lemma \ref{lem:ConeConvex} $q^{\phi(t)}_{x,z}\in C$ for all $t\in [0,1]$.
But this would imply that $q^y_{x,z}\in C$.
We proved already that $[q^y_{x,z}]\notin \mathbb{P}(C)$, so this is a contradiction.
 
 \medskip
 
Hence the four points \eqref{eq:OrderPoints} are cyclically ordered for this choice of a triple $(x,y,z)$.

\end{proof}

\begin{thm}
\label{thm:H2toH1}

Let $\rho:\Gamma_g \to \Sp(4,\mathbb{R})$ be a $\{1,2\}$-Anosov representation that satisfies property $H_2$. 
The representation $\rho$ must satisfy property $H_{1}$.
\end{thm}

\begin{proof}

Let $\rho$ be a $\lbrace 1,2\rbrace$-Anosov representation that satisfies property $H_2$. By Theorem \ref{thm:MaximalH_n}, $\rho$ is maximal. Because of Lemma \ref{lem:Order4Points}, there exist a triple of distinct points $(x,y,z)$ in $\partial \Gamma_g$ such that the quadruple \eqref{eq:OrderPoints} is cyclically ordered. This is for instance the case in the right part of Figure \ref{fig:IllustrationCircles}.

\medskip

Let $(u,v,w)$ be a triple of distinct points on $\partial \Gamma_g$ that is oriented as $(x,y,z)$. These two triples are joined by a continuous path in the space of disjoint triples in $\partial \Gamma_g$. Along this path, the cross ratio of the points \eqref{eq:OrderPoints} is defined and cannot vanish, because of Lemma \ref{lem:DistinctPoints}. Hence the cross ratio of these points stays negative. In particular for every triple of distinct points $(u,v,w)$ that are oriented in the circle $\partial\Gamma_g$ as $(x,y,z)$, the following $4$ points are cyclically ordered:
$$(\;w_\rho^3\cap  u_\rho^2,\;\;(v_\rho^1\oplus w^2_\rho)\cap u^2_\rho,\;\;v_\rho^3\cap  u_\rho^2,\;\;u^1_\rho\;)$$

In particular ${w}^3_\rho\cap{u}^2_\rho\neq {v}^3_\rho\cap {u}^2_\rho$, therefore the following sum is direct :
$${w}^3_\rho\cap{u}^2_\rho+ {v}^3_\rho\cap {u}^2_\rho+u^0_\rho.$$

Since this expression is invariant if one exchanges $v$ and $w$, this holds for all triple $(u,v,w)$ of distinct points in $\partial \Gamma_g$. Therefore property $H_3$ holds for $\rho$. Finally, because of Proposition \ref{prop:H_n-k},   property $H_1$ holds for $\rho$.

\end{proof}

We end this section by presenting a Proposition that describes the behavior of $y\mapsto[q^y_{x,z}]$. This proposition is not used in the proof of the main theorem but it helps to understand Figure \ref{fig:IllustrationCircles}.

\begin{prop}
\label{prop:limitsProjection}
Let $\rho:\Gamma_g\to \Sp(4,\R)$ be a $\{1,2\}$-Anosov representation that satisfies property $H_2$. Let $x,z\in \partial \Gamma_g$ be two distinct points.

\begin{itemize}
\item[(i)] The limit in $\mathbb{P}(\mathcal{Q}(x^2_\rho))$ of $\left[q_{x,z}^y\right]$ when $y$ converges to $x$ is $\iota(x^1_\rho)\in \partial \mathbb{P}( \mathcal{Q}^+(x^2_\rho))$.

\item[(ii)] If moreover $\rho$ satisfies property $H_1$, the limit in $\mathbb{P}(\mathcal{Q}(x^2_\rho))$ of $\left[q_{x,z}^y\right]$ when $y$ converges to $z$ is $\iota(z^3_\rho\cap x^2_\rho)\in \partial \mathbb{P}( \mathcal{Q}^+(x^2_\rho))$.
\end{itemize}

\end{prop}

Because of Theorem \ref{thm:H2toH1}, it is actually not necessary to require property $H_1$ for part (ii).

\medskip

\begin{proof}

Because of Lemma \ref{lem:DescriptionDerivativeBondaryMap2}, there is a parametrization $\phi:\R\to\partial\Gamma_g \setminus \{z\}$ such that $t\mapsto q^{\phi(t)}_{x,z}$ is a $\mathcal{C}^1$ embedding. Let us assume that $\phi(0)=x$ and let $\dot{q}(t_0)$ be the derivative of $t\mapsto q^{\phi(t)}_{x,z}$ at $t=t_0$. Since $q^{x}_{x,z}=0$, one can write for $t$ close to $0$:

$$ q^{\phi(t)}_{x,z}=t\dot{q}(0)+t\epsilon(t),$$

with $\epsilon(t)\to 0$ when $t\to 0$.

This implies that the limit of $\mathbb{P}(q^{\phi(t)}_{x,z})$ when $t$ goes to $0$ is equal to $\mathbb{P}(\dot{q})=\iota\left((x^1_\rho\oplus z^2_\rho)\cap x^2_\rho\right)=\iota(x^1_\rho)$, which proves (i).

\medskip

Let's now prove (ii). A representation $\rho$ satisfies property $H_1$, if and only if it is $(1,2,3)$-hyperconvex in the sense of (\cite{pozzetti2020conformality} Definition 6.1). Hence (\cite{pozzetti2020conformality}, Theorem 7.1) implies that the hyperplane $y^1_\rho\oplus z^2_\rho$ converges to $z^3_\rho$ when $y$ converges to $z$. 

Therefore $[\dot{q}(t)]\in \mathbb{P}(\mathcal{Q}(x^2_\rho))$ converges to $\iota(z^3_\rho\cap x^2_\rho)$. Hence for any small closed cone $C$ in $\mathcal{Q}^+(x^2_\rho)$ containing $\iota(z^3_\rho\cap x^2_\rho)$ in its interior, there is an affine cone $C_0$ directed by $C$ such that for any $t$ big enough $q_{x,z}^{\phi(t)}\in C_0$. 

Since $q_{x,z}^y$ diverges when $y\to z$, then any subsequence of $[q_{x,z}^y]$ converges to a point in $\mathbb{P}(C)$. Hence for $t\to +\infty$, $[q_{x,z}^y]$ converges to $\iota(z^3_\rho\cap x^2_\rho)$. When $t\to -\infty$, a similar argument holds.

\end{proof}

\section{Hyperconvex representations.}
\label{sec:Hyperconvex}

Let $N\geq 2$ be an integer. Labourie introduced the notion of an hyperconvex representation.

\begin{defn}
A Borel Anosov representation $\rho$ is hyperconvex if for all distinct $x_1,\cdots,x_{N} \in \partial \Gamma_g$, the $N$ lines $(x_1)_\rho^1,(x_2)_\rho^1,\cdots,(x_{N})_\rho^1$ span the whole vector space $\R^N$.

A Borel Anosov representation $\rho$ is $\{a,b,c\}$-hyperconvex if for all $x,z,y\in \partial \Gamma_g$ distinct, then the following sum is direct :
$$x_\rho^a+ y_\rho^b+ z_\rho^c.$$

If $\rho$ is $\{a,b,c\}$-hyperconvex for all $1\leq a\leq b\leq c$ such that $a+b+c\leq 2n$, then we say that it is $3$-hyperconvex. 
\end{defn}

\begin{rem}
A representation is $\{a,b,c\}$-hyperconvex in this sense if and only if it is $(a,b,N-c)$-hyperconvex in the sense of \cite{pozzetti2020conformality} .
\end{rem}

The following theorem of Labourie \cite{LabourieHyperconvexity} will enable us to show hyperconvexity using property $H$ for any maximal and Borel Anosov representation in $\Sp(4,\mathbb{R})$.

\begin{thm}[{\cite[Lemma 7.1]{LabourieHyperconvexity}}]
\label{thm:3HypImpliesHyp}
Every Borel Anosov representation that satisfies property $H$ and that is $3$-hyperconvex is hyperconvex.
\end{thm}

Then the following theorem of Guichard \cite{GuichardHyperconvexity} will enable us to show that hyperconvex representation are Hitchin. This theorem is one part of the characterization of Hitchin representations by the hyperconvexity condition. The other part was proved by Labourie \cite{LabourieHyperconvexity}.

\begin{thm}[{\cite[Theorem 1]{GuichardHyperconvexity}}]
\label{thm:HypImplesHitchin}
Any Borel Anosov and hyperconvex representation $\rho:\Gamma_g\to \SL(N,\R)$ is Hitchin.

\end{thm}

Finally we can prove our main Theorem.

\begin{thm}
\label{thm:mainSp4V2}
Every representation $\rho:\Gamma_g\to \Sp(4,\R)$ that is maximal and Borel Anosov is Hitchin.

\end{thm}

\begin{proof}
Because of Theorem \ref{thm:MaximalH_n} and Theorem \ref{thm:H2toH1}, the representation $\rho$ must satisfy property $H_2$, $H_1$, and therefore $H_3$ by  Proposition \ref{prop:H_n-k}. Hence $\rho$ satisfies property $H$. 

\medskip

When $n=2$, property $H_1$ is equivalent to $\{1,1,2\}$-hyperconvexity for a representation $\rho$. Moreover if $a+b+c=4$ with $a,b,c\geq 1$ then $\{a,b,c\}$ is equal to $\{1,1,2\}$. Therefore the representation $\rho$ is $3$-hyperconvex. By Theorem \ref{thm:3HypImpliesHyp}, the representation $\rho$ is hyperconvex, and by Theorem \ref{thm:HypImplesHitchin} it is Hitchin.

\end{proof}

\newpage
\bibliographystyle{hamsalpha}
\bibliography{biblio}
\end{document}